\documentclass[11pt,a4paper,leqno]{amsart}
\usepackage{amssymb, amsmath, amscd}
\usepackage[all]{xy}
\DeclareMathOperator{\h}{\rm hes}

\DeclareMathOperator{\Hes}{\rm Hes}
\DeclareMathOperator{\Ric}{\rm Ric}
\DeclareMathOperator{\tr}{\rm tr}
\DeclareMathOperator{\Id}{\rm Id}

\newtheorem{theorem}{Theorem}[section]

\newtheorem{lemma}[theorem]{Lemma}

\newtheorem{remark}[theorem]{Remark}

\usepackage{soul}

\begin{document}
\title[Half conformally flat gradient Ricci almost solitons]
{Half conformally flat gradient Ricci almost solitons}
\author[Brozos-V\'azquez, Garc\'ia-R\'io, Valle-Regueiro]{M. Brozos-V\'azquez, E. Garc\'ia-R\'io, 
X. Valle-Regueiro}
\address{MBV: Departmento de Matem\'aticas, Escola Polit\'ecnica Superior, Universidade da Coru\~na, Spain}
\email{miguel.brozos.vazquez@udc.gal}
\address{EGR-XVR: Faculty of Mathematics,
University of Santiago de Compostela,
15782 Santiago de Compostela, Spain}
\email{eduardo.garcia.rio@usc.es $\,\,$ javier.valle@usc.es}
\thanks{Supported by projects GRC2013-045, MTM2013-41335-P and EM2014/009 with FEDER funds (Spain).}
\date{}

\begin{abstract}
The local structure of half conformally flat gradient Ricci almost solitons is investigated, showing that they are locally conformally flat in a neighborhood of any point where the gradient of the potential function is non-null. In opposition, if the gradient of the potential function is null, then the soliton is a steady traceless $\kappa$-Einstein soliton and is realized on the cotangent bundle of an affine surface.
\end{abstract}

\maketitle

\section{Introduction}
A triple $(M,g,f)$, where $(M,g)$ is a pseudo-Riemannian manifold and $f$ is a smooth function on $M$, is said to be a \emph{gradient Ricci soliton} if the following equation is satisfied:
\begin{equation}\label{eq:RicciSoliton}
\operatorname{Hes}_f+\rho=\lambda\, g\,
\end{equation}
for some $\lambda\in\mathbb{R}$, where $\rho$ denotes the Ricci tensor and
$\operatorname{Hes}_f$ denotes the Hessian of $f$.

The special significance of gradient Ricci solitons comes from the analysis of the fixed points of the Ricci flow: $\frac{\partial}{\partial t}g(t)=-2\rho(t)$. Ricci flat metrics  are genuine fixed points of the flow. Moreover, if $(M,g_0)$ is Einstein with $\rho_0=\frac{\tau_0}{\operatorname{dim}M}g_0$, then $g(t)=(1-2\frac{\tau_0}{dim M})g_0$, which shows that $g_0$ is a fixed point of the Ricci flow modulo homotheties. Generalizing this property, gradient Ricci solitons correspond to self-similar solutions of the Ricci flow. They are ancient solutions in the shrinking case ($\lambda>0$), eternal solutions in the steady case ($\lambda=0$), and immortal solutions in the expanding case ($\lambda<0$). 

Recently, a generalization of Equation \eqref{eq:RicciSoliton} has been considered in \cite{PRRS}, allowing $\lambda$ to be a smooth function on $M$. Thus $(M,g,f)$ is said to be a \emph{gradient Ricci almost soliton} if  Equation~\eqref{eq:RicciSoliton} is satisfied for some $\lambda\in\mathcal{C}^\infty(M)$. Since gradient Ricci almost solitons contain gradient Ricci solitons as a particular case, we say that the gradient Ricci almost soliton is \emph{proper} if the function $\lambda$ is non-constant. 
 
It is important to emphasize that, beyond being a generalization of Ricci solitons, some proper gradient Ricci almost solitons correspond to self-similar solutions of some geometric flows. The \emph{Ricci-Bourguignon flow} is given by the equation
\[
\frac{\partial}{\partial t}g(t)=-2(\rho(t)-\kappa\tau(t)\, g(t)),
\]
where $\kappa\in\mathbb{R}$ and $\tau$ denotes the scalar curvature. This flow can be seen as an interpolation between the Ricci flow and the Yamabe flow, which corresponds to the equation
$\frac{\partial}{\partial t}g(t)=-\tau(t)\, g(t)$.
We refer to  \cite{CCDMM} for a broad exposition on the Ricci-Bourguignon flow. 
Following \cite{CCDMM}, the self-similar solutions of this flow are called \emph{$\kappa$-Einstein solitons} and correspond to the equation
\[
 \operatorname{Hes}_f+\rho=(\kappa\,\tau+\mu)\, g\,
\]
for some $\kappa, \mu\in\mathbb{R}$. Hence, they are a special family of gradient Ricci almost solitons with soliton function $\lambda=\kappa\,\tau+\mu$ (see  \cite{CMMR, CaMa} for more information on $\kappa$-Einstein solitons). The steady case ($\mu=0$) will play a relevant role in our study.

Gradient Ricci almost solitons exhibit some similarities but also some striking differences when comparing with usual Ricci solitons. For example, no positive-definite K\"ahler manifold admits proper gradient Ricci almost solitons \cite{Ma} in contrast with the gradient Ricci soliton case. Proper gradient Ricci almost solitons are irreducible and, moreover, they are of constant non-zero curvature in the homogeneous Riemannian setting \cite{CL-FL-GR-VL}.

Classifying gradient Ricci almost solitons under curvature conditions is a natural problem which certainly may help in understanding the corresponding flows. Among the different curvature conditions, local conformal flatness is the most natural one, since  the Ricci tensor completely determines the curvature in that situation. The analysis carried out in \cite{CaMa} gives a classification of locally conformally flat gradient $\kappa$-Einstein solitons.

The purpose of this paper is to analyze half  conformally flat (i.e. self-dual or anti-self-dual) four-dimensional gradient Ricci almost solitons. The main tasks are to show the existence of proper examples and to describe their underline structure.
The results are obtained as a consequence of the almost soliton equation \eqref{eq:RicciSoliton}, where  geometric information of different nature is encoded: on the one hand information related to the curvature of $(M,g)$ is given through the Ricci tensor; on the other hand the level hypersurfaces of the potential function $f$ are involved by means of their second fundamental form. As a result, two cases, which are different in nature, are consider separately. The first one corresponds to non-degenerate level hypersurfaces ($\|\nabla f\|\neq 0$), whereas the second one corresponds to degenerate level hypersurfaces  ($\|\nabla f\|=0$) and gives rise to the socalled {\it isotropic} solitons.

Our main result shows that any half conformally flat four-dimensional gradient Ricci almost soliton is locally conformally flat if $\|\nabla f\|\neq 0$. However, the isotropic case ($\|\nabla f\|=0$) allows the existence of strictly half conformally flat proper gradient Ricci almost solitons, i.e. examples which are self-dual but not locally conformally flat. All of them are realized as the cotangent bundle of an affine surface equipped with a modified Riemannian extension. Moreover, they are steady traceless $\kappa$-Einstein solitons, i.e. $\mu=0$ and $\kappa=\frac{1}{4}$.

\begin{theorem}\label{main}
Let $(M,g,f)$ be a four-dimensional half conformally flat proper gradient Ricci almost soliton. 
\begin{enumerate}
\item
If $\|\nabla f\|\neq 0$, then $(M,g)$ is locally isometric to a warped product of the form $I\times_\varphi N$, where $I\subset \mathbb{R}$ and $N$ is of constant sectional curvature. Furthermore, $(M,g)$ is locally conformally flat. 
\item
If $\|\nabla f\|=0$, then $(M,g)$ is locally isometric to the cotangent bundle $T^*\Sigma$ of an affine surface $(\Sigma,D)$ equipped with a modified Riemannian extension of the form $g = \iota T\circ \iota\Id +g_D +\pi^*\Phi$, where $T$ is a $(1,1)$-tensor field and $\Phi$ is a symmetric $(0,2)$-tensor field on $\Sigma$. 

The potential function satisfies $f=\pi^*\hat f$ for some smooth function $\hat f$ on $\Sigma$ and is related with the soliton function $\lambda$ by $\lambda=\frac32 C e^{-f}$ for a constant $C$.
Moreover, $T$ is given by $T=C e^{-\hat f}\Id$ and $\Phi$ is given by $\Phi=\frac{ 2}{C}e^{\hat f}(\operatorname{Hes}^D_{\hat f}+2\rho^D_{sym})$.
\end{enumerate}
\end{theorem}

The paper is organized as follows. Some basic consequences of the gradient Ricci almost soliton equation are discussed in $\S$\ref{se:2-1}. Half conformal flatness is discussed at the purely algebraic level in $\S$\ref{se:2-2}. Walker metrics and modified Riemannian extensions are introduced in $\S$\ref{se:2-3} and used in Section~\ref{se:3} to prove Theorem \ref{main}. The analysis of the non-isotropic case is carried out in $\S$\ref{se:3-1} to prove Theorem~\ref{main}-(1), while the isotropic case is treated in $\S$\ref{se:3-2} to prove Theorem~\ref{main}-(2).

\section{Preliminaries}

Let $(M,g)$ be a pseudo-Riemannian manifold with Levi-Civita connection $\nabla$ and curvature tensor $\mathcal{R}(X,Y)=\nabla_{[X,Y]}-[\nabla_X,\nabla_Y]$. 
Let $\rho$ and $\tau$ denote the Ricci tensor and the scalar curvature given by
$\rho(X,Y)=\tr\{Z\mapsto\mathcal{R}(X,Z)Y\}$ and $\tau=\tr\{\Ric\}$ respectively, where $\Ric$ denotes the Ricci operator defined by $g(\Ric X,Y)=\rho(X,Y)$. Let $\operatorname{Hes}_f$ denote the Hessian tensor defined by $\operatorname{Hes}_f(X,Y)=(\nabla_X df)(Y)=XY(f)-(\nabla_XY)(f)$. We begin by exploring the first consequences of the almost soliton equation \eqref{eq:RicciSoliton}.

\subsection{Gradient Ricci almost solitons: consequences of the equation}\label{se:2-1}

Let $(M,g,f)$ be a gradient Ricci almost soliton defined by Equation \eqref{eq:RicciSoliton}, where $\lambda$ is a function on $M$. Tracing and taking divergences on  \eqref{eq:RicciSoliton}, one gets the following relations which extend well-known identities from the soliton case. We refer to  \cite{BBR-2014, BBR--2012, Barros-Ribeiro12, PRRS} for a detailed exposition.

\begin{lemma}\label{le:2-1}
Let $(M,g,f)$ be a gradient Ricci almost soliton. Then
\begin{enumerate}
\item[(1)] $\Delta f + \tau = n\lambda$ ,
\item[(2)] $\nabla\Delta f + \nabla\tau = \nabla\lambda n$,
\item[(3)] $\nabla \Delta f+ \operatorname{Ric}(\nabla f)+\frac{1}{2}\nabla\tau = \nabla \lambda$,
\item[(4)] $\nabla\tau =2\operatorname{Ric}(\nabla f)+2(n-1)\nabla\lambda$,
\item[(5)] $\mathcal{R}(X,Y,Z, \nabla f)
=d\lambda(X)g(Y,Z)- d\lambda(Y)g(X,Z)\\
\phantom{R(X,Y,\nabla f,Z)=}\,\,
-(\nabla_X \rho)(Y,Z)+(\nabla_Y \rho)(X,Z)$.
\end{enumerate}
\end{lemma}

Let $\mathcal{W}$ be the  Weyl conformal curvature tensor
$$
\begin{array}{rcl}
\mathcal{W}(X,Y,Z,T)&=&\mathcal{R}(X,Y,Z,T)\\
\noalign{\medskip}
&&
+\frac{\tau}{(n-1)(n-2)}\{g(X,Z)g(Y,T)-g(X,T)g(Y,Z)\}\\
\noalign{\medskip}
     & &+\frac{1}{(n-2)}\{ \rho(X,T)g(Y,Z)-\rho(X,Z)g(Y,T)\\
     \noalign{\medskip}
     &&\phantom{+\frac{1}{(n-2)}}
     \,+\rho(Y,Z)g(X,T)-\rho(Y,T)g(X,Z)\},
\end{array}
$$
and let $C$ denote the Cotton tensor
$$
\begin{array}{rcl}
C(X,Y,Z)&=&(\nabla_X \rho)(Y,Z)- (\nabla_Y \rho)(X,Z)\\
\noalign{\medskip}
&& -\displaystyle\frac{1}{2n-2}\left\{ d\tau(X)g(Y,Z)-d\tau(Y)g(X,Z)\right\}.
\end{array}
$$
Now, an immediate application of Lemma~\ref{le:2-1}-- (4) and Lemma~\ref{le:2-1}--(5) gives the following expression for the Weyl conformal tensor. 

\begin{lemma}\label{le:2-2}
Let $(M,g,f)$ be a gradient Ricci almost soliton. Then 
$$	\begin{array}{rcl}
\mathcal{W}(X,Y,Z,\nabla f)&=&-C(X,Y,Z)\\
\noalign{\medskip}
&&
		+\frac{\tau}{(n-1)(n-2)}\{g(X,Z)g(Y,\nabla f)-g(X,\nabla f)g(Y,Z)\}\\
		\noalign{\medskip}
		&&		
		+\frac{1}{(n-2)}\{\rho(Y,Z)g(X,\nabla f) -\rho(X,Z)g(Y,\nabla f)\}\\	
		\noalign{\medskip}
		&&		
		+\frac{1}{(n-1)(n-2)}\{\rho(X,\nabla f)g(Y,Z)-\rho(Y,\nabla f)g(X,Z)\}.
		\end{array}
		$$
\end{lemma}

\subsection{Half conformally flat structure of four-dimensional manifolds}\label{se:2-2}
We work at the purely algebraic level. Let $(V,\langle\,\cdot\,,\,\cdot\,\rangle)$ be an inner-product vector space and let 
$\ll\,\cdot\,,\,\cdot\,\gg$ be the induced inner product on the space of two-forms $\Lambda^2(V)$. For a given orientation
$vol_V$, the Hodge star operator $\ast:\Lambda^2(V)\rightarrow\Lambda^2(V)$ given by $\alpha\wedge\ast\beta =\ll\alpha,\beta\gg vol_V$ satisfies $\ast^2=\Id$ and induces a decomposition  $\Lambda^2(V)=\Lambda^2_+\oplus\Lambda^2_-$, where 
$\Lambda^2_+ =\{\alpha\in \Lambda^2: \ast \alpha=\alpha \}$ and $\Lambda^2_- =\{\alpha\in \Lambda^2: \ast \alpha=-\alpha \}$. $\Lambda^2_+$ is the space of self-dual and $\Lambda^2_-$ is the space of anti-self-dual two-forms.
Moreover, a change of orientation on $(V,\langle\,\cdot\,,\,\cdot\,\rangle)$ reverses the roles of $\Lambda^2_+$ and $\Lambda^2_-$.

Any algebraic curvature tensor $\mathcal{R}$ on $(V,\langle\,\cdot\,,\,\cdot\,\rangle)$ (i.e., a $(0,4)$-tensor on $V$ satisfying the symmetries of the curvature tensor) can be naturally considered as an endomorphism $\mathcal{R}:\Lambda^2(V)\to \Lambda^2(V)$. Let $\mathcal{W}:\Lambda^2(V)\to \Lambda^2(V)$ be the corresponding endomorphism associated to the Weyl conformal tensor. $\mathcal{W}$ decomposes under the action of $SO(V,\langle\,\cdot\,,\,\cdot\,\rangle)$ as $\mathcal{W}=\mathcal{W}_+\oplus \mathcal{W}_-$, where
$\mathcal{W}_+=\frac{1}{2}(\mathcal{W}+\ast \mathcal{W})$ is the self-dual and $\mathcal{W}_-=\frac{1}{2}(\mathcal{W}-\ast \mathcal{W})$ is the anti-self-dual Weyl curvature tensor. 

An algebraic curvature tensor $\mathcal{R}$ is said to be \emph{conformally flat} if $\mathcal{W}=0$ and $\mathcal{R}$ is said to be \emph{half conformally flat} if it is either \emph{self-dual} (i.e., $\mathcal{W}_-=0$) or \emph{anti-self-dual} (i.e., $\mathcal{W}_+=0$).

The following algebraic characterization of self-dual algebraic curvature tensors will be used in the proof of Theorem \ref{main}.

\begin{lemma}\label{lema:autodualidade}
Let $(V,\langle\,\cdot\,,\,\cdot\,\rangle)$ be an oriented four-dimensional inner product space of neutral signature.
\begin{enumerate}
\item[(i)]
An algebraic curvature tensor $\mathcal{R}$ is self-dual if and only if for any positively oriented orthonormal basis $\{e_1,e_2,e_3,e_4\}$ 
$$
\mathcal{W}(e_1,e_i,x,y)=\sigma_{ijk}\epsilon_j \epsilon_k \mathcal{W}(e_j,e_k,x,y),\quad \text{for any }\,\,  x,y\in V,
$$
for $i,j,k\in \{2,3,4\}$  and where $\sigma_{ijk}$  denotes the signature of the corresponding permutation.
\item[(ii)] 
An algebraic curvature tensor $\mathcal{R}$ is self-dual if and only if for a positively oriented pseudo-orthonormal basis
$\{t,u,v,w\}$ (i.e., the non-zero inner products are $\langle t,v\rangle=\langle u,w\rangle=1$)
and for every $x,y\in V$,
$$
\qquad \quad\mathcal{W}(t, v, x,y)=\mathcal{W}(u,w,x,y),\quad
\mathcal{W}(t,w,x,y)=0,\,\, 
\mathcal{W}(u,v,x,y)=0.
$$
\end{enumerate}
\end{lemma}
\begin{proof}
Let $\{e_1,e_2,e_3,e_4\}$ be an orthonormal basis of $(V,\langle\,\cdot\,,\,\cdot\,\rangle)$ such that $vol_V=e^1\wedge e^2\wedge e^3\wedge e^4$. Then
\[
\Lambda^2_{\pm}=\operatorname{span}\{e^1\wedge e^2\pm \epsilon_3\epsilon_4 e^3\wedge e^4,e^1\wedge e^3\mp \epsilon_2\epsilon_4 e^2\wedge e^4,e^1\wedge e^4\pm \epsilon_2\epsilon_3 e^2\wedge e^3\},
\]
where $\epsilon_i=\langle e_i,e_i\rangle$. Now, $\mathcal{W}_-=0$ if and only if 
$$
\begin{array}{l}
\displaystyle
\mathcal{W}(e^1\wedge e^2- \epsilon_3\epsilon_4 e^3\wedge e^4)=0,\quad
\mathcal{W}(e^1\wedge e^3+\epsilon_2\epsilon_4 e^2\wedge e^4)=0,\\
\noalign{\medskip}
\displaystyle
\mathcal{W}(e^1\wedge e^4-\epsilon_2\epsilon_3 e^2\wedge e^3)=0,
\end{array}
$$
so $\mathcal{W}(e_1,e_i)=\sigma_{ijk}\epsilon_j \epsilon_k \mathcal{W}(e_j,e_k)$ for $i,j,k\in\{2,3,4\}$ and Assertion~(i) follows.

%Further observe that if $\langle\,\cdot\,,\,\cdot\,\rangle$ is of neutral signature then the induced inner product $\ll\,\cdot\,,\,\cdot\,\gg$ on $\Lambda^2_{\pm}$ is Lorentzian. 
Now, if $\{t,u,v,w\}$ is a positively oriented pseudo-orthonormal basis such that the only non-zero inner products are given by $\langle t,v\rangle=\langle u,w\rangle=1$, then 
$$
e_1=\frac{1}{\sqrt{2}}(t-v),\quad
e_2=\frac{1}{\sqrt{2}}(t+v),\quad
e_3=\frac{1}{\sqrt{2}}(w-u),\quad
e_4=\frac{1}{\sqrt{2}}(w+u),
$$
is a positively oriented orthonormal basis and Assertion~(ii) follows from the previous case.
\end{proof}

Now we move to the differentiable setting. 
Let $(M,g)$ be a four-dimensional pseudo-Riemannian manifold. $(M,g)$ is said to be \emph{half conformally flat} if its curvature tensor is half conformally flat at each point. If $(M,g)$ is Lorentzian then 
it is half conformally flat if and only if it is locally conformally flat. Hence we restrict our attention to the Riemannian and the neutral signature cases.

\subsection{Modified Riemannian extensions}\label{se:2-3}
A \emph{Walker manifold} is  a pseudo-Riemannian manifold $(M,g)$ which has a null parallel  distribution $\mathfrak{D}$, i.e., the restriction of the metric tensor to $\mathfrak{D}$ is totally degenerate and $\mathfrak{D}$ is invariant by parallel transport ($\nabla\mathfrak{D}\subset\mathfrak{D}$). 
%Walker manifolds are the underlying structure in many pseudo-Riemannian situations without Riemannian counterpart (see \cite{walker-metrics}).

The existence of a null $2$-plane induces an orientation on each tangent space $T_pM$ as follows. For any basis $\{ u_1,u_2\}$ of $\mathfrak{D}_p$, all basis $\{ u_1,u_2,v_1,v_2\}$  of $T_pM$ determined by $\langle u_i,v_j\rangle=\delta_{ij}$ induce the same orientation on $T_pM$ \cite{Derd}.
Thus, in  Walker coordinates $(x^1,x^2,x_{1'},x_{2'})$ where the metric tensor is written as
\[
g= 2 dx^i \circ dx_{i'} +a_{ij}\, dx^{i} \circ dx^{j},
\] 
for $a_{ij}$ functions on $M$ \cite{Walker}, the two-form $dx_{1'}\wedge dx_{2'}$ in the null parallel distribution is self-dual. This fixes the orientation of the manifold and we take
this to be the orientation of any Walker manifold throughout this work.

In this paper we are interested in a particular family of Walker manifolds: modified Riemannian extensions. In order to introduce these manifolds next we briefly review some properties of cotangent bundles (see \cite{walker-metrics} and references therein).

Let $\Sigma$ be a manifold and $T^*\Sigma$ its cotangent bundle. We express any point $\bar{p}\in T^*\Sigma$ as a pair $\bar{p} = (p,\omega)$, with $\omega$ a one-form on $T_p \Sigma$. Thus $\pi : T^*\Sigma \to \Sigma$ defined by $\pi(p,\omega)=p$ is the natural projection. For any vector field $X$ on $\Sigma$, we define  $\iota X$, the evaluation map of $X$, to be the smooth function on $T^*\Sigma$ given by $\iota X(p,\omega) = \omega(X_p)$. 
Vector fields on $T^*\Sigma$ are characterized by their action on evaluation maps. For a vector field $X$ on $\Sigma$, its complete lift $X^C$ is the vector field on $T^*\Sigma$ defined by $X^C(\iota Z) = \iota[X,Z]$, for all vector fields $Z$ on $\Sigma$. Tensor fields of type $(0,s)$ on $T^*\Sigma$ are also determined by their action on complete lifts of vector fields on $\Sigma$. 

Let $(\Sigma,D)$ be a  torsion-free affine manifold. Define the \emph{Riemannian extension} of $(\Sigma,D)$ to be the neutral signature metric $g_D$ defined on $T^*\Sigma$ and determined by $g_D(X^C,Y^C)=-\iota(D_X Y+D_Y X)$. In many situations it is useful to consider a deformation of the Riemannian extension given by $g_{D,\Phi}=g_D+\pi^*\Phi$, where $\Phi$ is a symmetric $(0,2)$-tensor field  on $\Sigma$.

For any $(1,1)$-tensor field $T$ on $\Sigma$, its evaluation $\iota T$ is a one-form on $T^*\Sigma$ characterized by $\iota T(X^C)=\iota (T(X))$. If $T$ and $S$ are $(1,1)$-tensors on $\Sigma$ and $\Phi$ is a symmetric $(0,2)$-tensor field on  $\Sigma$, then the \emph{modified Riemannian extension} is the metric on $T^*\Sigma$ given by
\begin{equation}\label{eq:metrictensorofriemannextensions}
	g_{D,\Phi,T,S} = \iota T \circ \iota S +g_D +\pi^*\Phi.
\end{equation}
%where $\circ$ denotes the symmetric product.
Local coordinates $(x^1,\dots,x^n)$ on $\Sigma$ induce local coordinates $(x^1,\dots,x^n$, $x_{1'},\dots, x_{n'})$ on $T^*\Sigma$ so that $g_{D,\Phi,T,S}$ expresses as:
\begin{equation}\label{eq:mre}
g_{D,\Phi,T,S}= 2 dx^i \circ dx_{i'} +\left\{\frac{1}{2} x_{r'} x_{s'}(T^r_i S^s_j +T^r_j S^s_i)-2x_{k'} {}^D\Gamma_{ij}^k+\Phi_{ij}\right\}dx^i \circ dx^j, 
\end{equation}
where ${}^D\Gamma_{ij}^k$ are the Christoffel symbols of the affine connection $D$, $\Phi$ $=$ $\Phi_{ij}dx^i\otimes dx^j$, $T=T^r_idx^i\otimes\partial_{x^r}$ and $S=S^s_jdx^j\otimes \partial_{x^s}$.

Modified Riemannian extensions have been used in \cite{CL-GR-G-VL} to describe self-dual Walker manifolds as follows.

\begin{theorem}\label{thm-7.2}
A four-dimensional Walker metric is self-dual  if and only
if it is locally isometric to the cotangent bundle $T^*\Sigma$ of an affine
surface $(\Sigma,D)$, equipped with a metric tensor of the form
$$
g=\iota X(\iota \Id\circ\iota \Id)+ \iota T\circ\iota\Id +g_{D}+\pi^*\Phi\,
$$
where $X$, $T$, $D$ and $\Phi$ are a vector field, a
$(1,1)$-tensor field, a torsion-free affine connection and a
symmetric $(0,2)$-tensor field on $\Sigma$, respectively.
\end{theorem}

\section{Half conformally flat gradient Ricci almost solitons}\label{se:3}

We analyze separately the isotropic and the non-isotropic cases to prove Theorem \ref{main}.

\subsection{Non-isotropic case}\label{se:3-1}
First we will consider non-isotropic half conformally flat gradient Ricci almost solitons (i.e., $\|\nabla f\|\neq 0$). The fact that the level hypersurfaces of the potential function are non-degenerate hypersurfaces will be crucial to show that such a  soliton is necessarily locally conformally flat. Since any Riemannian gradient Ricci almost soliton is non-isotropic, the following result shows that any Riemannian half conformally flat gradient Ricci almost soliton is locally conformally flat.

\begin{theorem}\label{Th:3-1}
Let $(M,g,f)$ be a non-isotropic four-dimensional half conformally flat gradient Ricci almost soliton. Then $(M,g)$ is locally isometric to a warped product of the form $I\times_\varphi N$, where $I\subset \mathbb{R}$ and $N$ is of constant sectional curvature. Furthermore, $(M,g)$ is locally conformally flat.
\end{theorem}
\begin{proof}
Since we work at the local level, let $p\in M$ and orient $(M,g)$ on a neighborhood of $p$ so that it is self-dual. 
Now, since the Cotton tensor is a constant multiple of the divergence of the Weyl tensor:  
$C(x,y,z)=-2(\operatorname{div} W)(x,y,z)$, we use Lemma \ref{le:2-2} to express the self-duality condition 
$
\mathcal{W}(e_1,e_i,x,y)=\sigma_{ijk}\epsilon_j \epsilon_k \mathcal{W}(e_j,e_k,x,y)
$
of Lemma~\ref{lema:autodualidade}-(i) applied to $\nabla f$
as
\begin{equation}\label{eq:relation1}
\begin{array}{l}
\displaystyle
\tau\{g(e_i,\nabla f)e_1-g(e_1,\nabla f)e_i\} - \{\rho(e_i,\nabla f)e_1-\rho(e_1,\nabla f)e_i
\\
\noalign{\medskip}
\displaystyle
\phantom{\tau \{g(e_i,\nabla f)e_1\}}
+3g(e_i,\nabla f)\operatorname{Ric}(e_1)-3g(e_1,\nabla f)\operatorname{Ric}(e_i)\}\\
\noalign{\medskip}
\phantom{\frac{\tau}{6} \{}
\displaystyle
=\sigma_{ijk}\, \epsilon_j\epsilon_k\big(  
 \tau \{g(e_k,\nabla f)e_j-g(e_j,\nabla f)e_k\}
\\
\noalign{\medskip}
\displaystyle
\phantom{\tau \{g(e_i,\nabla f)e_1\}}
-\{\rho(e_k,\nabla f)e_j-\rho(e_j,\nabla f)e_k
\\
\noalign{\medskip}
\displaystyle
\phantom{\tau\{g(e_i,\nabla f)e_1-\}+} 
  +3 g(e_k,\nabla f)\operatorname{Ric}(e_j)-3 g(e_j,\nabla f)\operatorname{Ric}(e_k)\}
\big)
\end{array}
\end{equation}
for $i,j,k\in \{2,3,4\}$.

Since the gradient Ricci almost soliton is non-isotropic, normalize $\nabla f$ to be unit and complete it to an orthonormal frame 	
$\left\{E_1 := \frac{\nabla f}{\|\nabla f\|},E_2,E_3,E_4\right\}$.
Let $i,j,k\in\{ 2,3,4\}$ henceforth. Equation~\eqref{eq:relation1} gives rise to
\begin{equation}\label{eq:4-2}
\begin{array}{l}
-\tau g(E_1,\nabla f)g(E_i,Z)
			+3 \rho(E_i,Z)g(E_1,\nabla f)\\	
\noalign{\medskip}
\qquad +\rho(E_1,\nabla f)g(E_i,Z)-\rho(E_i,\nabla f)g(E_1,Z)\\
\noalign{\medskip}
=\sigma_{ijk}\epsilon_j \epsilon_k \{\rho(E_j,\nabla f)g(E_k,Z)-\rho(E_k,\nabla f)g(E_j,Z)\}.
		\end{array}
\end{equation}
Now, put $Z=E_1$ in \eqref{eq:4-2} to get $\rho(E_1,E_i)=0$ for all $i\in\{ 2,3,4\}$, which shows that $\nabla f$ is an eigenvector of the Ricci operator.
Next, set $Z=E_j$ in \eqref{eq:4-2} and use this fact to obtain
$$	
3\rho(E_i,E_j)g(E_1,\nabla f)
				=-\sigma_{ijk}\epsilon_j \epsilon_k\rho(E_k,\nabla f)g(E_j,E_j) =0, 
$$
from where it follows that $\rho(E_i,E_j)=0$ for all $i\neq j$.
Finally, setting $Z=E_i$ in \eqref{eq:4-2} one gets 
$$
\tau g(E_1,\nabla f)g(E_i,E_i)-3 \rho(E_i,E_i)g(E_1,\nabla f) 
				-\rho(E_1,\nabla f)g(E_i,E_i)
				= 0,
$$
and thus 
$3\epsilon_i\rho(E_i,E_i)=\tau-\epsilon_1 \rho(E_1,E_1)$.

Hence the Ricci operator $\Ric$ diagonalizes on the basis $\{E_1,\dots,E_4\}$ and, moreover, it has at most two distinct eigenvalues, one of multiplicity one corresponding to the eigenvector $E_1$.

Now, the Ricci almost soliton equation \eqref{eq:RicciSoliton} shows that 
\[
\Hes_f(E_i,E_i)=\lambda g(E_i,E_i)-\rho(E_i,E_i)=\left(\lambda-\frac{\tau-\epsilon_1 \rho(E_1,E_1)}3\right) g(E_i,E_i),
\]
and thus the level hypersurfaces of $f$ are totally umbilical. 
Since the one-dimensional distribution $\operatorname{span}\{ E_1\}$ is totally geodesic one has that $(M,g)$ decomposes locally as a twisted product $I\times_\varphi N$ (see \cite{ponge-reckziegel}). Since the Ricci tensor is diagonal (in particular $\rho(E_1,E_i)=0$) one has that the twisted product reduces to a warped product \cite{manolo-eduardo}. Finally, since $I\times_\varphi N$ is self-dual, it is necessarily locally conformally flat and the fiber $N$ is of constant sectional curvature (see \cite{bv-gr-vl-2}).
\end{proof}

\begin{remark}\label{re:non-isotropic}
\rm
Let $I\times_\varphi N=(I\times N, \epsilon dt^2\oplus g_N)$ be a warped product as in Theorem~\ref{Th:3-1}, where $\epsilon=\pm 1$ and $(N,g_N)$ has constant sectional curvature $c_N$. Since $f$, $\lambda$ and $\varphi$ are function of $t$, the almost Ricci soliton equations \eqref{eq:RicciSoliton} are obtained by direct computation (see formulas for warped products in \cite{Oneill}):
\[
\begin{array}{l}
f''(t)-\epsilon  \lambda (t)-\frac{3 \varphi ''(t)}{\varphi (t)}=0,\\
\noalign{\smallskip}
\varphi (t) \left(f'(t) \varphi '(t)-\varphi ''(t)\right)-\epsilon  \lambda (t) \varphi
   (t)^2-2 \varphi '(t)^2+2 c_N  \epsilon=0.
\end{array}
\]
From the first equation $\lambda (t)=\epsilon f''(t)-\frac{3\epsilon \varphi ''(t)}{\varphi (t)}$ and substituting in the second equation one gets:
\[
-\varphi (t)^2 f''(t)+\varphi (t) \varphi '(t) f'(t) +2 \varphi (t) \varphi''(t)-2 \varphi '(t)^2+2 c_N  \epsilon=0.
\]
Note that this is a linear ODE for the potential function $f$, so for any warped product $I\times_\varphi N$ as above there exist globally defined solutions for $f$ and thus one can construct, in general, proper gradient Ricci almost solitons. 

Locally conformally flat Riemannian gradient Ricci almost solitons were studied in \cite{Ca}, whereas examples of gradient Ricci almost solitons on warped product manifolds of the form $I\times_\varphi N$, where $(N,g_N)$ is Einstein have been constructed in \cite[Example 2.5]{PRRS}.
\end{remark}

\subsection{Isotropic case}\label{se:3-2}

In contrast with the non-isotropic case, the level hypersurfaces of the potential function are now degenerate hypersurfaces.

\begin{lemma}\label{Th:GRAS then Walker}
Let $(M,g,f)$ be a four-dimensional isotropic gradient Ricci almost soliton. If $(M,g)$ is half conformally flat, then $(M,g)$ is locally a Walker manifold and  $\lambda=\frac{1}{4}\tau$.
\end{lemma}
\begin{proof}	
We firstly determine the structure of the Ricci operator. Since $\nabla f\neq 0$ but $g(\nabla f,\nabla f)=0$, complete it to a local pseudo-orthonormal frame $\mathcal{B}=\left\{\nabla f, U, V, W\right\}$, i.e. the only non-zero components of $g$ are $g(\nabla f,V)$ $=$ $g(U,W)$ $=$ $1$.

Since $g(\nabla f,\nabla f)=0$, one has that $0= \nabla_X g(\nabla f,\nabla f)=2g(\nabla_X \nabla f,\nabla f)=2g(\nabla_{\nabla f} \nabla f,X)$, and thus $\h_f(\nabla f)=0$, where $\h_f$ is the Hessian operator defined by $g(\h_f(X),Y)=\operatorname{Hes}_f(X,Y)$. 
Hence, it follows from the almost soliton equation \eqref{eq:RicciSoliton}  that $\nabla f$ is an eigenvector of the Ricci operator: 
$\Ric(\nabla f)=\lambda \nabla f$.
		
Now we use the self-duality characterization given in Lemma~\ref{lema:autodualidade}-$(ii)$, which with respect to the pseudo-orthonormal frame above reads as
$$
\begin{array}{l}
\mathcal{W}(\nabla f, V, X,Y)=\mathcal{W}(U,W,X,Y), \\
\noalign{\medskip}
\mathcal{W}(\nabla f,W,X,Y)=0,\qquad
\mathcal{W}(U,V,X,Y)=0,
\end{array}
$$
for any vector fields $X$ and $Y$.
Setting $Y=\nabla f$ in the first equation above and using Lemma~\ref{le:2-2} together with $\Ric(\nabla f)=\lambda \nabla f$
 one has
$$
0=\mathcal{W}(\nabla f, V, X,\nabla f)-\mathcal{W}(U,W,X,\nabla f)=
\frac{1}{6}(\tau-4\lambda)g(\nabla f,X)
$$
for any vector field $X$.
Hence $\tau=4\lambda$. 
Now, setting $Y=\nabla f$ in the third equation above: $\mathcal{W}(U,V,X,Y)=0$, it follows from Lemma~\ref{le:2-2} that 
$$
0=\mathcal{W}(U,V,X,\nabla f)=
\frac{1}{6}(\tau-\lambda)g(U,X)
-\frac{1}{2} \rho(U,X)
=\frac{1}{2}(\lambda g(U,X)
-\rho(U,X))
$$
for all $X$. This shows that $U$ is also an eigenvector of the Ricci operator associated to the eigenvalue $\lambda$: $\Ric(U)=\lambda U$. 

Finally, setting $X=V$ in the second equation above: $\mathcal{W}(\nabla f,W,X,Y)=0$, and using Lemma \ref{le:2-2} again, one has
$$
0=\mathcal{W}(V,Y,W,\nabla f)=	-\frac{1}{2}\{\lambda g(Y,W)
-\rho(Y,W) +\rho(V,W)g(Y,\nabla f)\}
$$
for all $Y$. Hence it follows that  $\rho(W,W)=\rho(\nabla f,W)=0$.

In summary, the Ricci operator expresses in the basis $\mathcal{B}=\left\{\nabla f,  U,V, W\right\}$ as 
$$
\Ric=\left(	\begin{array}{cccc}
	\lambda & 0 & \alpha & \beta\\
	0 & \lambda & \beta & 0\\
	0 & 0 & \lambda & 0\\
	0 & 0 & 0 & \lambda	
	\end{array}\right), 
$$	
where $\alpha$ and $\beta$ are functions on $M$.

Now set $\mathfrak{D}=\operatorname{span}\left\{\nabla f, U\right\}$, which is a two-dimensional null distribution.
We already showed that 
$g(\nabla_X \nabla f,\nabla f)=0$. Moreover, a similar argument using that $g(U,U)=0$ gives
$g(\nabla_X U,U)=0$ for all $X$. 
On the other hand, since  $\Ric(U)=\lambda U$, it follows from Equation~\eqref{eq:RicciSoliton}  that $\h_f(U)=0$. Now $g(U,\nabla f)=0$ shows that 
$$
g(\nabla_X U,\nabla f)=-g(U,\nabla_X\nabla f)=-\Hes_f(U,X)=0\,,
$$
for all $X$. Hence, since $\mathfrak{D}^\perp=\mathfrak{D}$ and 
$$
g(\nabla_X \nabla f,\nabla f)=0,\,\, g(\nabla_X U,U)=0,\,\, g(\nabla_X U,\nabla f)=0\,\,
\text{and}\,\,  g(U,\nabla_X\nabla f)=0
$$
we have that $\nabla \mathfrak{D}\subset\mathfrak{D}$ and $(M,g)$ is locally a Walker manifold.
\end{proof}

The choice of orientation did not play any role in our previous discussion. However, for a Walker manifold the orientation is induced by the orientation of the null two-dimensional distribution $\mathfrak{D}$. The next two lemmas show that the self-dual and the anti-self-dual conditions are not interchangeable in this context.

\begin{lemma}
\label{le:proper}
Let $(M,g)$ be a four-dimensional Walker manifold. If $(M,g,f)$ is an anti-self-dual isotropic gradient Ricci almost soliton then $(M, g, f)$ is a steady gradient Ricci soliton.
\end{lemma}
\begin{proof}
It was observed in \cite{DR-GR-VL} that if the self-dual Weyl curvature $W^+$ of a Walker manifold vanishes, then the scalar curvature is identically zero, and hence Lemma \ref{Th:GRAS then Walker} shows that $\lambda=0$.
\end{proof}

\begin{lemma}\label{le:3-3}
Let $(M,g)$ be a four-dimensional Walker manifold. If $(M,g,f)$ is a self-dual isotropic gradient Ricci almost soliton, then $(M,g)$ is locally isometric to the cotangent bundle $T^*\Sigma$ of an affine surface $(\Sigma,D)$ equipped with a modified Riemannian extension of the form $g_{D,\Phi,T,\Id} = \iota T\circ\iota\Id +g_D +\pi^*\Phi$. Moreover:
\begin{enumerate}
\item
The potential function satisfies $f=\pi^*\hat f$ for some smooth function $\hat f$ on $\Sigma$ and $f$ is related with the soliton function  by $\lambda=\frac32 C e^{-f}$ for a constant $C$.
\item 
The $(1,1)$-tensor field $T$ is given by $T=C e^{-\hat f}\Id$ and the $(0,2)$ symmetric tensor field  $\Phi$ is given by $\Phi=\frac{ 2}{C}e^{\hat f}(\operatorname{Hes}^D_{\hat f}+2\rho^D_{sym})$.
\end{enumerate}
\end{lemma}

\begin{proof}
It was shown in \cite{CL-GR-G-VL} that a four dimensional Walker manifold is self-dual if and only
if it is locally isometric to the cotangent bundle $T^*\Sigma$ of an affine
surface $(\Sigma,D)$, with metric tensor
\begin{equation}
\label{eq:3-4-1}
g=\iota X(\iota \Id\circ\iota \Id)+ \iota T\circ\iota\Id +g_{D}+\pi^*\Phi\,
\end{equation}
where $X$, $T$, $D$ and $\Phi$ are a vector field, a
$(1,1)$-tensor field, a torsion-free affine connection and a
symmetric $(0,2)$-tensor field on $\Sigma$, respectively. 
Now a direct calculation of the almost soliton equation \eqref{eq:RicciSoliton} gives:
\begin{equation}
\label{eq:3-4-2}
\begin{array}{rcl}
(\operatorname{Hes}_f +\rho -\lambda g)(\partial_{x_{1'}},\partial_{x_{1'}}) &=&\partial^2_{x_{1'}x_{1'}}f =0,\\
\noalign{\medskip}
(\operatorname{Hes}_f +\rho -\lambda g)(\partial_{x_{1'}},\partial_{x_{2'}}) &=&\partial^2_{x_{1'}x_{2'}}f=0,\\
\noalign{\medskip}
(\operatorname{Hes}_f +\rho -\lambda g)(\partial_{x_{2'}},\partial_{x_{2'}}) &=&\partial^2_{x_{2'}x_{2'}}f =0,
\end{array}
\end{equation}
which shows that the potential function $f$ of any gradient Ricci almost soliton must be of the form
$f=\iota\xi+\pi^*\hat{f}$ for some vector field $\xi$ on $\Sigma$ and some smooth function $\hat{f}$ on $\Sigma$.

We firstly show that no gradient Ricci almost soliton may exist if the vector field $X$ in \eqref{eq:3-4-1} does not vanish. Assume $X\neq 0$ at some point $p\in\Sigma$ and specialize local coordinates on $\Sigma$ so that $X=\partial_{x^1}$. 
Let $T=T_i^j dx^i\otimes\partial_{x^j}$ and $\xi=\xi^l\partial_{x^l}$ be the local expressions of the tensor field $T$ and the vector field $\xi$.  Then a straightforward calculation using that $\lambda=\frac{1}{4}\tau$ shows that 
$$
\begin{array}{l}
(\operatorname{Hes}_f +\rho -\frac{1}{4}\tau g)(\partial_{x^{1}},\partial_{x_{1'}}) =
\frac{1}{4}\left\{  2(T_1^1-T_2^2+2(x_{1'}+\partial_{x^1}\xi^1) \right.\\
\noalign{\medskip}
 \phantom{(\operatorname{Hes}_f +\rho -\lambda g)(\partial_{x^{1}}}
-2\xi^1(3 x_{1'}^2 +2 x_{1'}T_1^1+x_{2'}T_1^2-2{}^D\Gamma_{11}^1)\\
\noalign{\medskip}
 \phantom{(\operatorname{Hes}_f +\rho -\lambda g)(\partial_{x^{1}}}
\left. +\xi^2(4 x_{1'} x_{2'} + 2 x_{1'}T_2^1+x_{2'}(T_1^1+T_2^2)-4{}^D\Gamma_{12}^1)
 \right\},
\end{array}
$$
where ${}^D\Gamma_{ij}^k$ are the Christoffel symbols of the affine connection $D$.
This is a polynomial in the variables $x_{1'}$, $x_{2'}$, so all coefficientes must vanish and one has that $\xi=0$. The previous expression reduces to
$$
(\operatorname{Hes}_f +\rho -\frac{1}{4}\tau g)(\partial_{x^{1}},\partial_{x_{1'}}) =
x_{1'}+\frac{1}{2}(T_1^1-T_2^2).
$$
Since the above expression is not identically zero, we conclude that there do not exist gradient Ricci almost solitons if $X$ does not vanish.

Hence any self-dual metric \eqref{eq:3-4-1} reduces to 
\begin{equation}
\label{eq:3-4-3}
g_{D,\Phi,T,\Id}=\iota T\circ\iota\Id+g_{D}+\pi^*\Phi\,
\end{equation}
for some $(1,1)$-tensor field $T$, an affine connection $D$ and a
symmetric $(0,2)$-tensor field $\Phi$ on $\Sigma$. Moreover, Equation \eqref{eq:3-4-2} shows that the potential function of any Ricci almost soliton must be of the form $f=\iota\xi+\pi^*\hat{f}$ for some vector field $\xi$ on $\Sigma$ and $\hat{f}\in\mathcal{C}^\infty(\Sigma)$.
In order to show that the potential function $f$ is the pull-back of a smooth function $\hat f$ on $\Sigma$, assume the vector field $\xi\neq 0$ at some point $p\in\Sigma$. Let $(x^1,x^2)$ be adapted local coordinates so that $\xi=\partial_{x^1}$ and hence $f=x_{1'}+\pi^*\hat{f}$. 
Computing again expressions of the almost soliton equation one has
$$
\begin{array}{l}
\!\!(\operatorname{Hes}_f +\rho -\frac{1}{4}\tau g)(\partial_{x^{1}},\partial_{x_{1'}})\!=
\frac{1}{2}\{(1-2x_{1'})T_1^1-x_{2'}T_1^2-T_2^2 \}+{}^D\Gamma_{11}^1, \\
\noalign{\medskip}
\!\!(\operatorname{Hes}_f +\rho -\frac{1}{4}\tau g)(\partial_{x^{2}},\partial_{x_{2'}})\!=
 \frac{1}{2}\{(4-2x_{1'})T_2^2-x_{2'}T_1^2-(1+x_{1'})T_1^1\}+{}^D\Gamma_{12}^2,
\end{array}
$$
it follows that $T_1^1=T_1^2=T_2^2=0$. Since the scalar curvature of any metric given by \eqref{eq:3-4-3} satisfies $\tau=3(T_1^1+T_2^2)=3\tr\{ T\}$, one has that $\tau=0$ and thus $\lambda=0$, which shows that $(M,g,f)$ is a steady gradient Ricci soliton. These were studied in \cite{BV-GR}.

Hence in what follows set $M=T^*\Sigma$ with metric given by \eqref{eq:3-4-3}, $f=\pi^*\hat f$ and $\lambda =\frac{3}{4}\tr \{T\}$. Once more, considering the almost soliton equation one has 
$$
\begin{array}{l}
(\operatorname{Hes}_f +\rho -\frac{1}{4}\tau g)(\partial_{x^{1}},\partial_{x_{1'}}) =
\frac{1}{2}(T_1^1-T_2^2) =0,\\
\noalign{\medskip}
(\operatorname{Hes}_f +\rho -\frac{1}{4}\tau g)(\partial_{x^{1}},\partial_{x_{2'}}) = T_1^2 =0,
\\
\noalign{\medskip}
(\operatorname{Hes}_f +\rho -\frac{1}{4}\tau g)(\partial_{x^{2}},\partial_{x_{1'}}) = T_2^1 =0,
\end{array}
$$
from where it follows that $T$ is a multiple of the identity. Set $T=\phi(x^1,x^2)\Id$
for some smooth function $\phi$ on $\Sigma$. 
Moreover, a explicit calculation of the components of the almost soliton equation shows that they are polynomials on the fiber coordinates of the form
$$
\begin{array}{l}
(\operatorname{Hes}_f +\rho -\frac{1}{4}\tau g)(\partial_{x^{1}},\partial_{x^1}) =
\frac{2}{3}x_{1'}(\phi\partial_{x^1}\hat f+\partial_{x^1}\phi) + \cdots,\\
\noalign{\medskip}
(\operatorname{Hes}_f +\rho -\frac{1}{4}\tau g)(\partial_{x^{2}},\partial_{x^2}) = 
\frac{2}{3}x_{1'}(\phi\partial_{x^2}\hat f+\partial_{x^2}\phi) + \cdots\,.
\end{array}
$$
Hence $\phi=Ce^{-\hat f}$ and the potential function $f=\pi^*\hat f$ determines the $(1,1)$-tensor field $T$ as $T=C e^{-\hat f}\Id$, which shows that the potential function of the soliton $f$ and the soliton function $\lambda$ are related by $\lambda=\frac32 C e^{-\pi^*\hat f}$.

Finally, a long but straightforward calculation shows that
$$
\begin{array}{l}
(\operatorname{Hes}_f +\rho -\frac{1}{4}\tau g)(\partial_{x^{i}},\partial_{x^j}) =
\frac{C}2 e^{-\hat f}\Phi_{ij}-(\operatorname{Hes}^D_{\hat f}+2\rho^D_{sym})(\partial_{x^{i}},\partial_{x^j}) 
\end{array}
$$
for $i,j\in\{ 1,2\}$, where $\operatorname{Hes}^D_{\hat f}=Dd\hat f$ is the Hessian tensor on $(\Sigma,D)$ and $\rho^D_{sym}(X,Y)=\frac{1}{2}\{\rho^D(X,Y)+\rho^D(Y,X)\}$ is the symmetric part of the Ricci tensor of $(\Sigma,D)$. Now the result follows.
\end{proof}

\begin{remark}\label{re:3-5}
\rm
The previous lemma not only gives the local structure of self-dual isotropic gradient Ricci almost solitons, but also provides a construction method for steady traceless $\kappa$-Einstein solitons. 

Indeed, for any affine connection $D$ and any given function $\hat f$ on $\Sigma$, consider the $(1,1)$-tensor field $T$ and the symmetric $(0,2)$-tensor field $\Phi$ defined by
$$
T=C e^{-\hat f}\Id,\qquad \Phi=\frac{2}{C}e^{\hat f}(\operatorname{Hess}^D_{\hat f}+2\rho^D_{sym}).
$$
Then the modified Riemannian extension $g_{D,\Phi,T,\Id}$ is a neutral signature metric on $T^*\Sigma$ with scalar curvature $\tau= 6Ce^{-\pi^*\hat f}$ and
$(T^*\Sigma$, $g_{D,\Phi,T,\Id}$, $f=\pi^*\hat f$) becomes a  gradient Ricci almost soliton with $\lambda=\frac{1}{4}\tau$.
\end{remark}

\begin{remark}\label{re:3-6}
\rm
Let $(\Sigma,D)$ be an affine surface and $\hat f$ a constant function on $\Sigma$. Then $\operatorname{Hess}^D_{\hat f}=0$ and the modified Riemannian extension in Remark~\ref{re:3-5} is Einstein. This fact was already observed in \cite[Theorem 2.1]{CL-GR-G-VL}, but the construction given in Remark \ref{re:3-5} generalizes it to the context of gradient Ricci almost solitons.
\end{remark}

\begin{remark}\label{re:3-7}\rm
Self-dual gradient Ricci solitons were investigated in \cite{BV-GR} showing that non-conformally flat examples may only occur in the isotropic case. Such a soliton is necessarily steady and it is locally isometric to the cotangent bundle of an affine surface $(\Sigma, D)$ equipped with a Riemannian extension $g_{D,\Phi}$. Moreover, the potential function $f$ is obtained as the pull-back of a function $\hat f$ on $\Sigma$ which must satisfy an affine gradient Ricci soliton equation: $\operatorname{Hess}^D_{\hat f}+2\rho^D_{sym}=0$.

The existence of affine gradient Ricci solitons is a restrictive condition. For instance, a compact homogeneous affine surface admits an affine gradient Ricci soliton if and only if the Ricci tensor is of rank one (see \cite{BVGRG}). However, any affine surface gives rise to a proper gradient Ricci almost soliton on its cotangent bundle.

In opposition to the proper almost soliton case considered in this paper, the soliton case $(\lambda=0)$ is independent of the deformation tensor $\Phi$ but the potential function $f=\pi^*\hat f$ is restricted by the affine gradient Ricci soliton equation. Recall that in the proper almost soliton case, Theorem \ref{main} shows that while the deformation tensor $\Phi$ is completely determined, there is no restriction on the potential function since $f=\pi^*\hat f$ for arbitrary $\hat f\in\mathcal{C}^\infty(\Sigma)$.
\end{remark}

\end{document}